\documentclass[11pt]{amsart}
\usepackage{tabularx,booktabs,tikz}
\usepackage{caption}
\usepackage{amsmath}
\usepackage{amsfonts}

\usepackage{amscd}
\usepackage{amsthm}
\usepackage{amssymb} \usepackage{latexsym}
\usepackage{eufrak}
\usepackage{euscript}
\usepackage{epsfig}
\usepackage{graphics}
\usepackage{array}
\usepackage{enumerate}
\usepackage{dsfont}
\usepackage{color}
\usepackage{wasysym}
\usepackage{hyperref}
\usepackage{pdfsync}

\newcommand{\bel}[1]{\begin{equation}\label{#1}}

\newcommand{\be}{\begin{equation}}

\newcommand{\ba}{\begin{eqnarray}}
\newcommand{\ea}{\end{eqnarray}}

\newcommand{\qe}{\end{equation}}
\newcommand{\R}{{\mathbb R}}
\newcommand{\N}{{\mathbb N}}
\newcommand{\Z}{{\mathbb Z}}

\newcommand{\HH}{\mathcal{H}}

\newcommand{\y}{(\Omega)}

\newcommand{\diam}{\mathrm{diam}}

\newcommand{\Hmm}[1]{\leavevmode{\marginpar{\tiny%
$\hbox to 0mm{\hspace*{-0.5mm}$\leftarrow$\hss}%
\vcenter{\vrule depth 0.1mm height 0.1mm width \the\marginparwidth}%
\hbox to
0mm{\hss$\rightarrow$\hspace*{-0.5mm}}$\\\relax\raggedright #1}}}

\newtheorem{theorem}{Theorem}[section]
\newtheorem{con}[theorem]{Conjecture}
\newtheorem{lemma}[theorem]{Lemma}
\newtheorem{corollary}[theorem]{Corollary}
\newtheorem{definition}[theorem]{Definition}

\newtheorem{remark}[theorem]{Remark}

\newtheorem{prop}[theorem]{Proposition}

\newcommand{\tm}{\begin{theorem}}
\newcommand{\tmd}{\end{theorem}}
\newcommand{\co}{\begin{corollary}}
\newcommand{\cod}{\end{corollary}}
\newcommand{\prp}{\begin{prop}}
\newcommand{\prpd}{\end{prop}}

\begin{document}

\title[Discrete harmonic functions on infinite penny graphs]{Discrete harmonic functions on infinite penny graphs}


\author{Bobo Hua}
\email{bobohua@fudan.edu.cn}
\address{Bobo Hua: School of Mathematical Sciences, LMNS, Fudan University, Shanghai 200433, China; Shanghai Center for Mathematical Sciences, Fudan University, Shanghai 200433, China}


\begin{abstract} 

In this paper, we study discrete harmonic functions on infinite penny graphs. For an infinite penny graph with bounded facial degree, we prove that the volume doubling property and the Poincar\'e inequality hold, which yields the Harnack inequality for positive harmonic functions. Moreover, we prove that the space of polynomial growth harmonic functions, or ancient solutions of the heat equation, with bounded growth rate has finite dimensional property.
\end{abstract}
\maketitle



\par
\maketitle

\bigskip


\section{Introduction}


In geometric graph theory, penny graphs are contact graphs of unit circles. They are the graphs formed by arranging pennies in a non-overlapping way on the plane, making a vertex for each penny, and making an edge for each two pennies that touch. Finite penny graphs are extensively studied in the literature, to cite a few \cite{Harborth74,Pollack85,Kupitz94,Pach95,Pach96,Csizmadia98,Pisanski00,Hlin01,Swanepoel09,Cerioli11,Eppstein18}. In this paper, we study infinite penny graphs and the function theory defined on them.

Let $\{C_i\}_{i=1}^\infty$ be a collection of circles of radius $\frac12,$ which are the boundaries of open disks $\{B_i\}_{i=1}^\infty,$ called pennies, in $\R^2,$ such that $B_i\cap B_j=\emptyset,$ for any $i\neq j.$ Let $(V,E)$ be the contact graph of this configuration, i.e. $V=\{v_i\}_{i=1}^\infty$
where each $v_i$ represents the circle $C_i,$ and $\{v_i,v_j\}\in E$ if and only if $C_i$ and $C_j$ are tangent to each other. We call $(V,E)$ the penny graph of the configuration $\{C_i\}_{i=1}^\infty.$ Note that it is a locally finite, simple, undirected graph. For any $x\in V,$ we denote by $\deg(x)$ the vertex degree of $x.$ For a penny graph, by the Euclidean geometry,
$$\deg(x)\leq 6,\quad \forall x\in V.$$  

This graph has a natural geometric realization.
We write $$\phi:V\to \R^2, \phi(v_i)=c_i,$$ where $c_i$ is the center of $C_i.$ For each edge $\{v_i,v_j\},$ we draw a segment $\overline{c_ic_j}$ connecting the centers, and write $\phi(\{v_i,v_j\})=\overline{c_ic_j}.$ This yields a geometric embedding of $(V,E)$ into $\R^2,$ and induces a CW complex structure, denote by $G=(V,E,F).$ Here $F$ is the set of faces, which corresponds to connected components of the complement of the embedding image of $(V,E).$ For any $\sigma\in F,$ we denote by $\deg(\sigma)$ the facial degree of $\sigma.$ In this paper, a \emph{penny graph} refers to the CW complex structure $G=(V,E,F)$ induced by a configuration of non-overlapping open disks of diameter $1.$  A penny graph $G=(V,E,F)$ is called connected if the 1-skeleton $(V,E)$ is connected.
In this paper, we only consider connected penny graphs.

For a general planar graph, the combinatorial curvature at vertices resembles the Gaussian curvature for smooth surfaces \cite{MR0279280,Stone76,Gromov87,Ishida90}. We studied geometric and analytic properties of infinite planar graphs
with nonnegative curvature in \cite{HJLcrelle15}. In particular, we proved the volume doubling property and the Poincar\'e inequality, see Definition~\ref{defi:vd} and Definition~\ref{defi:poi}, for such graphs. For a graph $(V,E),$ the combinatorial Laplacian $\Delta$ is defined as, for any $f:V\to \R,$
$$\Delta f(x):=\sum_{y} (f(y)-f(x)),\quad \forall x\in V,$$ where the summation is taken over neighbours of $x.$ A function $f:V\to \R$ is called harmonic if $\Delta f\equiv 0.$ See Figure~\ref{fig1} and Figure~\ref{fig2} for some harmonic functions on penny graphs. Various Liouville type theorems for harmonic functions were proved via the volume doubling property and the Poincar\'e inequality. Following the strategy in \cite{HJLcrelle15}, we study geometric and analytic properties of infinite planar graphs in this paper.
 \begin{figure}[htbp]
 \begin{center}
   \begin{tikzpicture}
    \node at (0,0){\includegraphics[width=0.5\textwidth]{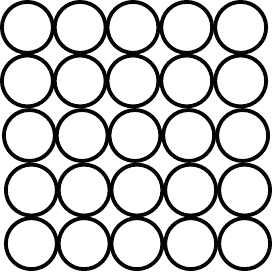}};
    \node at (0,   0){\Large $0$};
    \node at (1.25,  0){\Large $1$};
    \node at (2.5,  0){\Large $2$};
    \node at (-1.25,  0){\Large $-1$};
    \node at (-2.5,  0){\Large $-2$};
    
    \node at (0,   1.25){\Large $1$};
    \node at (1.25,  1.25){\Large $2$};
    \node at (2.5,  1.25){\Large $3$};
    \node at (-1.25,  1.25){\Large $0$};
    \node at (-2.5,  1.25){\Large $-1$};
    
        \node at (0,   2.5){\Large $2$};
    \node at (1.25,  2.5){\Large $3$};
    \node at (2.5,  2.5){\Large $4$};
    \node at (-1.25,  2.5){\Large $1$};
    \node at (-2.5,  2.5){\Large $0$};

    \node at (0,   -1.25){\Large $-1$};
    \node at (1.25,  -1.25){\Large $0$};
    \node at (2.5, -1.25){\Large $1$};
    \node at (-1.25, -1.25){\Large $-2$};
    \node at (-2.5,  -1.25){\Large $-3$};

        \node at (0,   -2.5){\Large $-2$};
    \node at (1.25,  -2.5){\Large $-1$};
    \node at (2.5,  -2.5){\Large $0$};
    \node at (-1.25,  -2.5){\Large $-3$};
    \node at (-2.5,  -2.5){\Large $-4$};

   \end{tikzpicture}
  \caption{A harmonic function $f(x_1,x_2)=x_1+x_2,$ $(x_1,x_2)\in \R^2,$ on the square lattice.}
\label{fig1}
 \end{center}
\end{figure}
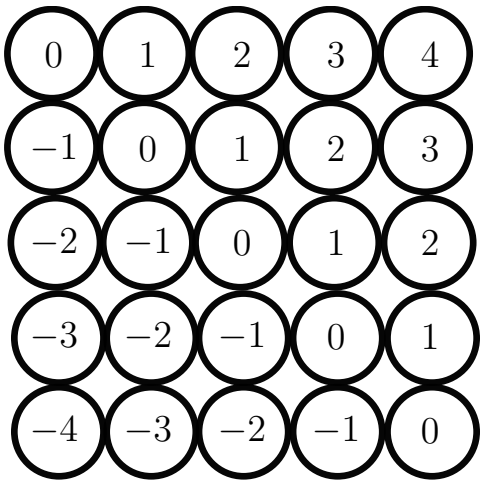

\begin{figure}[htbp]
 \begin{center}
   \begin{tikzpicture}
    \node at (0,0){\includegraphics[width=0.55\textwidth]{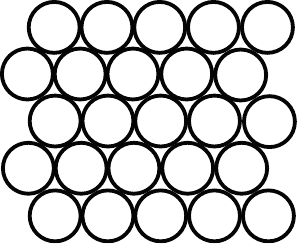}};
   \node at (0.3,   0){\Large $0$};
    \node at (1.55,  0){\Large $1$};
    \node at (2.8,  0){\Large $2$};
    \node at (-0.95,  0){\Large $-1$};
    \node at (-2.2,  0){\Large $-2$};
    
       \node at (0.3,   2.23){\Large $3$};
    \node at (1.55,  2.23){\Large $4$};
    \node at (2.8,  2.23){\Large $5$};
    \node at (-0.95,  2.23){\Large $2$};
    \node at (-2.2,  2.23){\Large $1$};

           \node at (-0.35,   1.11){\Large $1$};
    \node at (0.9,  1.11){\Large $2$};
    \node at (2.15,  1.11){\Large $3$};
    \node at (-1.6,  1.11){\Large $0$};
    \node at (-2.85,  1.11){\Large $-1$};

        \node at (-0.35,   -1.11){\Large $-2$};
    \node at (0.9,  -1.11){\Large $-1$};
    \node at (2.15,  -1.11){\Large $0$};
    \node at (-1.6,  -1.11){\Large $-3$};
    \node at (-2.85,  -1.11){\Large $-4$};

           \node at (0.3,   -2.23){\Large $-3$};
    \node at (1.55,  -2.23){\Large $-2$};
    \node at (2.8,  -2.23){\Large $-1$};
    \node at (-0.95,  -2.23){\Large $-4$};
    \node at (-2.2,  -2.23){\Large $-5$};

   \end{tikzpicture}
  \caption{A harmonic function on the triangular lattice.}
\label{fig2}
 \end{center}
\end{figure}
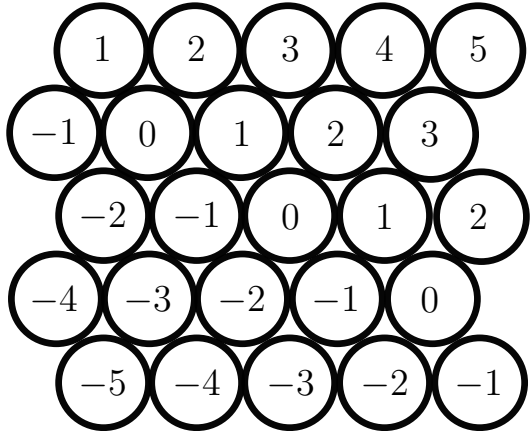

We denote by $\Z^2$ the integer lattice graph in $\R^2$ with the set of vertices $\Z^2=\{(x_1,x_2):x_1,x_2\in \Z\}$ and the set of edges $\{\{w,z\}\in E: w,z\in \Z^2, |w-z|=1\}.$ For any graph $(V,E),$ we denote by $d_{V}$ the combinatorial distance on the graph. We prove that any infinite penny graph with bounded facial degree is quasi-isometric to $\Z^2,$ see Definition~\ref{def:quasi}. 
\tm\label{thm:m1}Let $G=(V,E,F)$ be an infinite penny graph such that $\sup_{\sigma\in F}\deg(\sigma)<\infty.$ Then the metric spaces $(V,d_{V})$ and $(\Z^2,d_{\Z^2})$ are quasi-isometric. 
\tmd

Since the volume doubling property and the Poincar\'e inequality are quasi-isometric invariants \cite{MR1363211}, we prove the following theorem.
\tm\label{thm:main1}Let $G=(V,E,F)$ be an infinite penny graph such that $\sup_{\sigma\in F}\deg(\sigma)\leq D<\infty.$ Then the volume doubling property and the Poincar\'e inequality hold with constants depending on $D.$
\tmd

For an infinite penny graph with bounded facial degree, by Theorem~\ref{thm:main1}, we prove the following analytic properties for harmonic functions: the recurrence of the simple random walk, the Harnack inequality for positive harmonic functions, and the finite dimensionality property for harmonic functions (and ancient solutions to the heat equation) of polynomial growth.

\textbf{Acknowledgements.} This is dedicated to Professor J\"urgen Jost on the occasion of his 65th birthday. The author is supported by NSFC, no.11831004 and no. 11926313.

 \section{Preliminaries}
For a simple, undirected graph $(V,E),$ two vertices $x,y$ are called neighbors, denote by $x\sim y,$ if there is an edge connecting $x$ and $y.$ The vertex degree of $x,$ $\deg(x),$ is the number of neighbors of $x.$  The combinatorial distance $d$ on the graph is defined as, for any $x,y\in V$ and $x\neq y,$ $$d(x,y):=\inf\{n\in\N: \exists\{x_{i}\}_{i=1}^{n-1}\subset V, x\sim x_1\sim\cdots \sim x_{n-1}\sim y \}.$$  For any $x\in V$ and $R\geq 0,$ we denote by $$B_R(x):=\{y\in V: d(y,x)\leq R\}$$ the ball of radius $R$ centered at $x.$ We denote by $|\cdot |$ the counting measure on $V,$ i.e. for any $\Omega\subset V,$ $|\Omega|$ denotes the number of vertices in $\Omega.$
 The triple $(V,d,|\cdot|)$ is a discrete metric measure space. For any $\Omega\subset V,$ we denote by $$\delta\Omega:=\{y\in V\setminus \Omega: \exists x\in \Omega, y\sim x\}$$ the vertex boundary of $\Omega.$ We write $\overline{\Omega}:=\Omega\cup \delta\Omega.$ For any function $f:\overline{\Omega}\to\R,$ the Laplacian of $f$ is defined as
$$\Delta f(x):=\sum_{y:y\sim x} (f(y)-f(x)),\quad \forall x\in \Omega.$$

For simplicity, for any $A,B\in \R^2$ we denote by $\overline{AB}$ the segment connecting $A$ and $B,$ and by $|AB|:=|A-B|$ the Euclidean distance between $A$ and $B.$

Let $G=(V,E,F)$ be a connected penny graph, with the embedding $\phi:\{V,E\}\to \R^2.$ Note that for any two vertices $x,y,$ $x\sim y$ if and only if $|\phi(x)\phi(y)|=1,$ and $x\not\sim y$ if and only if $|\phi(x)\phi(y)|>1.$ The embedding map $\phi$ determines the embedding of vertices and edges, hence induces the embedding of faces. For simplicity, we denote by $\phi(\sigma)$ the embedding image of $\sigma \in F$. By the Euclidean geometry, $\phi(\sigma)$ is a polygonal domain, whose boundary is a piecewise linear curve consisting of embedding images of edges, see $(1)$ in Proposition~\ref{prop:red} below. The facial degree of $\sigma,$ $\deg(\sigma)\in \N\cup\{\infty\},$ is the number of interior angles (or corners) of the polygonal domain $\phi(\sigma).$ For a general infinite penny graph, it is possible that there are some faces of infinite facial degree. In this paper, we always consider infinite penny graphs with the property that each face has finite facial degree.




 Note that for a penny graph $G=(V,E,F),$ for any $x,y\in V,$ the combinatorial distance $d(x,y)$ is given by the induced distance between $\phi(x)$ and $\phi(y)$ on the embedding image of the graph $(V,E).$

By the Euclidean geometry, we derive some properties for polygonal domains corresponding to embedding images of faces.
\begin{prop}\label{prop:red} Let $A,B,C,D$ be four points in $\R^2$ such that $|AB|=|CD|=1,$ $\min\{|AC|,|AD|,|BC|,|BD|\}>1.$ Then 
\begin{enumerate}
\item $\overline{AB}\cap \overline{CD}=\emptyset,$ and
\item
$$\min_{\substack{x\in \overline{AB},\\y\in \overline{CD}}}|xy|=\min\left\{\min_{\substack{x\in \{A,B\},\\y\in \overline{CD}}}|xy|,\min_{\substack{x\in \overline{AB},\\y\in \{C,D\}}}|xy|\right\}.$$
\end{enumerate}
\end{prop}
\begin{proof} For the first assertion, suppose that it is not true. Pick a point $H\in \overline{AB}\cap \overline{CD}.$ Without loss of generality, we assume that
$$|AH|=\min\{|AH|,|BH|\}, |CH|=\min\{|CH|,|DH|\}.$$ Hence $|AH|\leq \frac12$ and $|CH|\leq \frac12.$ By the triangle inequality, this yields that 
$$|AC|>1\geq |AH|+|CH|.$$ This is a contradiction. We prove the first assertion.

For the second assertion, choose $x_0\in \overline{AB}, y_0\in \overline{CD}$ such that \begin{equation}\label{eq:minimal}|x_0y_0|=\min_{\substack{x\in \overline{AB},\\y\in \overline{CD}}}|xy|.\end{equation} By the above claim, $x_0\neq y_0.$ If $x_0\in\{A,B\}$ or $y_0\in\{C,D\},$ then the result holds. Hence it suffices to consider the case that $x_0\in \overline{AB}\setminus \{A,B\}$ and $y_0\in \overline{CD}\setminus \{C,D\}.$ By the minimality in \eqref{eq:minimal},
$\overline{x_0y_0}$ is perpendicular to $\overline{AB}$ and $\overline{CD}.$ So that $\overline{AB}$ and $\overline{CD}$ are parallel. Hence one can move $\overline{x_0y_0}$ along the direction of $\overline{AB}$ to reach the ends of $\overline{AB}$ or $\overline{CD}.$ That is, one can find $x_0'\in \overline{AB}, y_0'\in \overline{CD}$ such that $x_0'\in\{A,B\}$ or $y_0'\in\{C,D\},$ and
$$|x_0'y_0'|=|x_0y_0|.$$ This proves the proposition.
\end{proof}

\begin{prop}\label{prop:triangle} Let $A,B,C$ be three points in $\R^2$ such that $|AB|=1,$ $\min\{|AC|,|BC|\}>1.$ Then 
$$\min_{x\in \overline{AB}}|xC|\geq\frac{\sqrt{3}}{2}.$$
\end{prop}
\begin{proof} It suffices to consider that the angles $\angle CAB$ and $\angle CBA$ are acute. Otherwise, $\min_{x\in \overline{AB}}|xC|=\min\{|AC|,|BC|\}\geq 1$ which yields the result. 

Let $a = |BC|, b = |CA|, c = |AB|$ be the side lengths. By the cosine rule, 
$$\cos \angle ACB=\frac{a^2+b^2-c^2}{2ab}\geq \frac{1}{2}.$$ Hence $\angle ACB\leq \frac{\pi}{3}.$
We denote by $CD$ the altitude from the vertex $C$ to the base $AB,$ where $D$ is the foot of the altitude.  Obviously, $|CD|=\min_{x\in \overline{AB}}|xC|.$ In the right triangles, $\Delta ACD$ and $\Delta BCD,$ $$\frac{|CD|}{\min\{a,b\}}=\max\{\cos \angle ACD,\cos \angle BCD\}.$$ Note that $\angle ACD+\angle BCD=\angle ACB\leq \frac{\pi}{3}.$ This yields that
$$\frac{|CD|}{\min\{a,b\}}\geq\cos\frac{\pi}{6}=\frac{\sqrt 3}{2}.$$
\end{proof}

By these propositions, we have the following lemma.
\begin{lemma}\label{lem:nonadj} Let $A,B,C,D$ be four points in $\R^2$ such that $|AB|=|CD|=1,$ $\min\{|AC|,|AD|,|BC|,|BD|\}>1.$ Then for any $x\in \overline{AB}, y\in \overline{CD},$
$$|xy|\geq \frac{\sqrt 3}{2}.$$
\end{lemma}
\begin{proof} This follows from Proposition~\ref{prop:red} and Proposition~\ref{prop:triangle}.
\end{proof}

For a connected penny graph, for any face $\sigma\in F$ with finite facial degree, the boundary of embedding image of $\sigma,$ $\partial \phi(\sigma),$ 
is a piecewise linear curve. For any $x,y\in \partial \phi(\sigma),$ we denote by $d_{\partial\sigma}(x,y)$ the induced distance between $x$ and $y$ in $\partial \phi(\sigma),$ i.e. the shortest length of paths connecting $x$ and $y$ in
$\partial \phi(\sigma).$ One readily sees that 
\begin{equation}\label{eq:oneface} d_{\partial\sigma}(x,y)\leq \frac{1}{2}\deg(\sigma), \quad \forall x,y\in \partial \phi(\sigma).
\end{equation} 
The diameter of $\phi(\sigma)$ is defined as, 
$$\diam(\phi(\sigma)):=\sup_{w,z\in \phi(\sigma)}|wz|.$$ One easily sees that 
$$\diam(\phi(\sigma))=\sup_{w,z\in \partial\phi(\sigma)}|wz|.$$ By \eqref{eq:oneface}, we have
\begin{equation}\label{eq:diam}\diam(\phi(\sigma))\leq \sup_{w,z\in \partial\phi(\sigma)}d_{\partial\sigma}(w,z)\leq \frac{1}{2}\deg(\sigma).\end{equation}

\begin{lemma}\label{lem:bilip}For any $\sigma\in F,$ $x,y\in \partial \phi(\sigma),$ 
$$Cd_{\partial\sigma}(x,y) \leq |xy|\leq d_{\partial\sigma}(x,y),$$ where $C=\frac{1}{2\deg(\sigma)}.$
\end{lemma}
\begin{proof} The upper bound estimate of $|xy|$ is trivial, since the segment realizes the minimal length among the curves connecting their end-points.

For the lower bound estimate of $|xy|,$ we choose edges $e_x$ and $e_y$ incident to $\sigma$ such that 
$$x\in \overline{\phi(e_x)}, y\in \overline{\phi(e_y)},$$ where $\overline{(\cdot)}$ denotes the closure of $(\cdot).$ Note that the choices of $e_x$ and $e_y$ may not be unique, but it suffices to choose some. If $e_x=e_y,$ then $|xy|=d_{\partial\sigma}(x,y),$ which yields the result.
So that we assume that $e_x\neq e_y.$ Then we have the following two cases:

\emph{Case 1.} $\overline{\phi(e_x)}\cap  \overline{\phi(e_y)}= \emptyset.$ Then by Lemma~\ref{lem:nonadj},
$$|xy|\geq \frac{\sqrt 3}{2}.$$ Hence by \eqref{eq:oneface},
$$|xy|\geq \frac{\sqrt 3}{\deg(\sigma)}d_{\partial\sigma}(x,y).$$

\emph{Case 2.}  $\overline{\phi(e_x)}\cap  \overline{\phi(e_y)}\neq \emptyset.$ Let $z=\overline{\phi(e_x)}\cap  \overline{\phi(e_y)}.$ Then by the cosine rule in the triangle $\Delta xyz,$ noting that $\angle xzy\geq \frac{\pi}{3},$
\begin{eqnarray}|xy|^2&=&|xz|^2+|yz|^2-2|xz||yz|\cos\angle xzy\nonumber\\
&\geq &|xz|^2+|yz|^2-|xz||yz|\geq \frac14(|xz|+|yz|)^2.\label{eq:hh1}
\end{eqnarray} 
We divide it into subcases.

\emph{Case 2.1.} $|xz|+|yz|<\frac12.$ We claim that $d_{\partial\sigma}(x,y)=|xz|+|yz|.$ In fact, if the minimal path connecting $x$ and $y$ passing through another edge, different from $e_x$ and $e_y,$ on the boundary, then the length is at least $1.$ The minimality contradicts with $|xz|+|yz|<\frac12.$ This is impossible. So that the minimal path only involves the points on $e_x$ and $e_y.$ This yields the claim. Hence by \eqref{eq:hh1}, $|xy|\geq \frac12 d_{\partial\sigma}(x,y).$

\emph{Case 2.2.} $|xz|+|yz|\geq \frac12.$ Then by \eqref{eq:hh1},
$$|xy|\geq \frac14\geq \frac{1}{2\deg(\sigma)}d_{\partial\sigma}(x,y).$$

Combining the above results, we prove the lemma.

\end{proof}

\section{Harmonic functions on penny graphs}

We recall the definition of the quasi-isometry between metric spaces.
\begin{definition}[\cite{BBI01}]\label{def:quasi} We say two metric spaces $(X,d_X)$ and $(Y,d_Y)$ are quasi-isometric if there is a map $\psi:X\to Y,$ $a\geq 1,C>0$ such that 
\begin{enumerate}[(a)]
\item for any $x_1,x_2\in X,$ 
$$a^{-1} d_X(x_1,x_2)-C \leq d_Y(\psi(x_1),\psi(x_2))\leq a d_X(x_1,x_2)+C,$$ and
\item for any $y\in Y,$ there exists $x\in X$ such that $d_Y(y,\psi(x))<C.$
\end{enumerate} We call $\psi$ the quasi-isometry map.
\end{definition}

\begin{theorem}\label{thm:quasi} Let $G=(V,E,F)$ be an infinite penny graph such that $\sup_{\sigma\in F}\deg(\sigma)\leq D<\infty.$ Then for any $x,y\in V,$ $$C(D)d(x,y)\leq |\phi(x)\phi(y)|\leq d(x,y),$$ where $C(D)=\frac{1}{2D}.$ Moreover, the metric space $(V,d)$ and $(\R^2,|\cdot|)$ are quasi-isometric.
\end{theorem}
\begin{proof} For the first assertion, the upper bound estimate of $|\phi(x)\phi(y)|$ is trivial. For the lower bound estimate of $|\phi(x)\phi(y)|,$ we consider the segment $\overline{\phi(x)\phi(y)}$ connecting $\phi(x)$ and $\phi(y).$
Then there exist mutually distinct $z_k\in\overline{\phi(x)\phi(y)}\setminus\{\phi(x),\phi(y)\},$ $1\leq k\leq K$, and some faces $\sigma_k,$ $1\leq k\leq K+1,$ such that $\overline{z_{k-1}z_{k}}\subset \overline{\phi(\sigma_k)}$ for $1 \leq k\leq K+1$ and $$|\phi(x)\phi(y)|=\sum_{k=1}^{K+1}|z_{k-1}z_{k}|,$$ where we write $z_0=\phi(x)$ and $z_{K+1}=\phi(y).$ For each face $\phi(\sigma_k),$ $1\leq k\leq K+1$, Lemma~\ref{lem:bilip} yields that there exists a path $\gamma_{z_{k-1}z_k}\subset \partial \phi(\sigma_k)$ from $z_{k-1}$ to $z_k$ such that \begin{equation}\label{eq:pa1}C(D)L(\gamma_{z_{k-1}z_k})\leq |z_{k-1}z_k|,\end{equation} where $L(\cdot)$ denotes the length of the curve. Consider the concatenation of paths $\{\gamma_{z_{k-1}z_k}\}_{k=1}^{K+1},$
$$\gamma=\gamma_{z_{0}z_{1}}\circ\gamma_{z_{1}z_{2}}\circ\cdots\circ\gamma_{z_{K}z_{K+1}},$$ which starts at $z_0$ and ends at $z_{K+1}.$ Summing over $k$ in \eqref{eq:pa1}, we have
$$C(D)L(\gamma)\leq \sum_{k=1}^{K+1}|z_{k-1}z_{k}|=|\phi(x)\phi(y)|.$$ Hence by the definition of $d(x,y),$ $$C(D)d(x,y)\leq |\phi(x)\phi(y)|.$$ This proves the first assertion.

For the second assertion, it suffices to prove $(b)$ in Definition~\ref{def:quasi}. Note that for any $z\in \R^2,$ there exists $\sigma\in F$ such that $z\in \overline{\phi(\sigma)}.$ Pick a vertex $x\in \partial \sigma.$ By the diameter estimate \eqref{eq:diam}, 
$$|z\phi(x)|\leq \diam(\phi(\sigma))\leq \frac{D}{2}.$$ This proves the second assertion.

So that the embedding map $\phi:V\to \R^2$ is the quasi-isometry map between the metric space $(V,d)$ and $(\R^2,|\cdot|).$ We prove the theorem.
\end{proof}

It is well-known that 
$(\R^2,|\cdot|)$ and $(\Z^2,d_{\Z^2}),$ where $d_{\Z^2}$ is the combinatorial distance on $\Z^2$, are quasi-isometric.
Now we ready to prove Theorem~\ref{thm:m1}.
\begin{proof}[Proof of Theorem~\ref{thm:m1}] Since the quasi-isometric relation between metric spaces is an equivalent relation, see \cite{BBI01}, the theorem follows from Theorem~\ref{thm:quasi} and the fact that $(\R^2,|\cdot|)$ and $(\Z^2,d_{\Z^2})$ are quasi-isometric. 
\end{proof}


\begin{definition}\label{defi:vd}
For a graph $G=(V,E),$ we say that it satisfies
\begin{itemize}
  \item the \emph{volume doubling property} if there exists $C_1>0$ such that \begin{equation}\label{eq:1}|B_{2R}(x)|\leq C_1|B_{R}(x)|, \quad \forall x\in V, R>0;\end{equation}
  \item the \emph{weak relative volume comparison} if there exist $C_1>0, \nu>0$ such that
  \begin{equation}\label{eq:2}\frac{|B_{R}(x)|}{|B_{r}(x)|}\leq C_2 \left(\frac{R}{r}\right)^\nu, \quad \forall x\in V, R>r>0;\end{equation}


\end{itemize}
\end{definition}
Note that \eqref{eq:1} implies that \eqref{eq:2} for $\nu=\log_2 C_1.$ So that these two conditions are in fact equivalent up to some constants.

 \begin{definition}\label{defi:poi}
   We say that a graph $G=(V,E)$ satisfies the \emph{Poincar\'e inequality} if there exists a constant $C$ such that  for any $x\in V, R>0,$ and any function $f:B_{2R}\to \R,$ 
  \begin{equation}\label{eq:poincare}\sum_{y\in B_R(x)}|f(y)-\overline{f}|^2\leq CR^2\sum_{y,z\in B_{2R}(x)}|f(y)-f(z)|^2,\end{equation} where $\overline{f}=\frac{1}{|B_R(x)|}\sum_{B_R(x)}f.$
\end{definition}

By Theorem~\ref{thm:m1}, we prove that the volume doubling property and the Poincar\'e inequality holds for infinite penny graphs with bounded facial degree.
\begin{theorem}\label{thm:vdb}Let $G$ be an infinite penny graph with facial degree bounded above by $D.$ Then the volume doubling property \eqref{eq:1} holds for some $C_1(D)$ and the weak relative volume comparison holds for some $C_2(D)$ and $\nu=2$. Moreover, there are positive constants $C_3(D)$ and $C_4(D)$ such that
\begin{equation}\label{eq:qua}C_3 R^2\leq |B_R(x)|\leq C_4 R^2,\quad \forall R\geq 1.\end{equation}
\end{theorem}
  \begin{proof} By Theorem~\ref{thm:m1}, since $G$ is quasi-isometric to $\Z^2.$ The volume growth properties for $G$ follow from those for $\Z^2.$
  \end{proof}
  \co Let $G$ be an infinite penny graph with bounded facial degree.  The simple random walk on $G$ is recurrent. Any positive superharmonic function on $G$ is constant.
  \cod
  \begin{proof} By Theorem~\ref{thm:vdb}, the graph $G$ has quadratic volume growth, \eqref{eq:qua}, which yields that the simple random walk on $G$ is recurrent. The second statement is equivalent to the first one.
  \end{proof}
  
  \begin{theorem}\label{thm:poinc}Let $G$ be an infinite penny graph with facial degree bounded above by $D.$ Then the Poincar\'e inequality \eqref{eq:poincare} holds for $C(D).$
\end{theorem}
\begin{proof} Note that the Poincar\'e inequality holds on $(\Z^2,d_{\Z^2}).$ It is well-known, \cite{MR1363211}, that the Poincar\'e inequality is a quasi-isometric invariant for graphs with bounded degree. This yields the result for $G.$
\end{proof}
 Now we can prove Theorem~\ref{thm:main1}.
 \begin{proof}[Proof of Theorem~\ref{thm:main1}] The result follows from the combination of Theorem~\ref{thm:vdb} and Theorem~\ref{thm:poinc}.
 \end{proof}

  By the Moser iteration on graphs, Delmotte \cite{Delmotte97} proved the Harnack inequality for positive harmonic functions. Using that result, we have the following corollary by Theorem~\ref{thm:main1}.
  \co Let $G$ be an infinite penny graph with facial degree bounded above by $D.$ Then there exists $C(D)$ such that for any positive harmonic function on $B_{2R}(x),$ $R\geq 1,$ we have $$\sup_{B_R(x)}f\leq C\inf_{B_R(x)}f.$$
  \cod

 
 \begin{remark} It is well-known, \cite{Delmotte99}, that for a graph with bounded degree, the following are equivalent:
 \begin{enumerate}
 \item The volume doubling property and the Poincar\'e inequality hold.
 \item The parabolic Harnack inequality holds.
 \end{enumerate} This yields that the parabolic Harnack inequality holds for infinite penny graphs with bounded facial degree.
 \end{remark}


Fix a point $x_0\in V.$  For a graph $G=(V,E)$ and any $k\geq 0,$ we denote by $$\HH^k(G):=\left\{u: \Delta u=0, |u(x)|\leq C(1+d(x,x_0))^k, \forall x\in V\right\}$$ the space of harmonic functions of polynomial growth whose growth rate is at most $k.$
Note that the above space is independent of the choice of $x_0.$

On a Riemannian manifold with nonnegative Ricci curvature, the finite dimensionality property of the space of harmonic functions of polynomial growth was conjectured by Yau, and was confirmed by Colding-Minicozzi \cite{ColdingMinicozzi97}, see also \cite{ColdingMinicozzi98,ColdingMinicozzi98Weyl,Li97}. Following these arguments,  Delmotte \cite{Delm97} proved the finite-dimensional property for harmonic functions of polynomial growth on graphs under the assumptions of the volume doubling property and the Poincar\'e inequality, see also \cite{HJLcrelle15}. By this result, we get the following corollary.
 \co Let $G$ be an infinite penny graph with facial degree bounded above by $D.$ Then there exists $C(D)$ such that $$\dim \HH^k(G)\leq Ck^2,\quad \forall k\geq 1.$$  
 \cod
We propose a conjecture on the dimension estimate.
\begin{con}\label{con:de}Let $G$ be an infinite penny graph with facial degree bounded above by $D.$ Then there exists $C(D)$ such that $$\dim \HH^k(G)\leq Ck,\quad \forall k\geq 1.$$
\end{con}



\begin{definition} We say that $u:V\times (-\infty,0]\to\R$ is an ancient solution to the heat equation if
$$\partial_t u(x,t)=\Delta u(x,t),\quad \forall x\in V, t\leq 0.$$
\end{definition}
Fix a point $x_0\in V.$ We denote by
$$\mathcal{P}^k(G):=\{\mathrm{ancient\ solution}\ u: |u(x,t)|\leq C(1+d(x,x_0)+\sqrt{|t|})^k, \forall x\in V,t\leq 0\}$$ the space of ancient solutions of polynomial growth with growth rate at most $k.$ 
Note that the above space doesn't depend on the choice of $x_0.$

Ancient solutions of polynomial growth were studied by many authors in Riemannian geometry, \cite{CalleMZ06,Callethesis,LinZhang17,ColdingM19}. Colding and Minicozzi  \cite{ColdingM19} prove a theorem to compare the dimension of the space of ancient solutions of polynomial growth with that of the space of harmonic functions of polynomial growth on Riemannian manifolds. By following their argument, we extended the result to graphs.
\begin{theorem}[\cite{Huaancient19}] Let $G=(V,E)$ be a graph with bounded degree and  of polynomial volume growth. Then for any $k\in \N,$
$$\dim\mathcal{P}^{2k}(G)\leq (k+1)\dim \HH^{2k}(G).$$
\end{theorem}

This yields the corollary for penny graphs.
\co Let $G$ be an infinite penny graph with bounded facial degree.
Then for any $k\geq 1,$$$\dim\mathcal{P}^{2k}(G)\leq C(D) k^3.$$
\cod
\begin{remark} If Conjecture~\ref{con:de} holds, then the above result will be improved to that for any $k\geq 1,$$$\dim\mathcal{P}^{2k}(G)\leq C(D) k^2.$$
\end{remark}


\bibliography{py1}
\bibliographystyle{alpha}

\bigskip
\bigskip

\bigskip

\end{document}